\documentclass{article}
\usepackage{amsfonts}
\usepackage{amsmath}
\usepackage{harvard}

\newtheorem{theorem}{Theorem}

\newtheorem{corollary}[theorem]{Corollary}

\newtheorem{example}[theorem]{Example}

\newtheorem{lemma}[theorem]{Lemma}

\newenvironment{proof}[1][Proof]{\noindent\textbf{#1.} }{\ \rule{0.5em}{0.5em}}

\input{tcilatex}

\begin{document}

\title{Integer powers of complex anti-tridiagonal matrices and some complex
factorizations}
\author{Durmu\c{s} Bozkurt\thanks{%
Corresponding author, dbozkurt@selcuk.edu.tr} \& H. K\"{u}bra Duru\thanks{%
hkduru@selcuk.edu.tr} \\
Department of Mathematics, Science Faculty of Sel\c{c}uk University}
\maketitle

\begin{abstract}
In this paper, we obtain a general expression for the entries of the $r$th
power of a certain $n\times n$ complex anti-tridiagonal matrix where if $n$
is odd, $r\in \mathbb{Z}$ or if $n$ is even, $r\in \mathbb{N}$. In addition,
we get the complex factorizations of Fibonacci polynomials, Fibonacci and
Pell numbers.
\end{abstract}

\section{Introduction}

\bigskip Arbitrary integer powers of a square matrix is used to solve some
difference equations, differential and delay differential equations and
boundary value problems.

Recently, the calculations eigenvalues and integer powers of
anti-tridiagonal matrices have been well studied in the literature. For
instance, Rimas [1-3] obtained the integer powers of anti-tridiagonal
matrices of odd and even order. \"{O}tele\c{s} and Akbulak [8-10]
generalized Rimas's the some results and obtained complex factorizations.
Guti\'{e}rrez [11] calculated the powers of complex persymmetric or
skew-persymmetric anti-tridiagonal matrices with costant anti-diagonals. For
details on the eigenvalues and the powers of tridiagonal and
anti-tridiagonal matrices, see [4-7, 12]. We obtained integer powers of the
tridiagonal matrix%
\begin{equation}
\widetilde{A}_{n}=\left\{ 
\begin{array}{l}
\widetilde{a}_{11}=\widetilde{a}_{nn}=a \\ 
\widetilde{a}_{12}=\widetilde{a}_{n,n-1}=2b \\ 
\widetilde{a}_{21}=\widetilde{a}_{n-1,n}=b \\ 
tridiag(-b,a,-b),\ other%
\end{array}%
\right.  \label{1}
\end{equation}

as

\begin{eqnarray}
\widetilde{a}_{ij}(r) &=&\frac{1}{2n-2}\left( \lambda
_{2}^{r}T_{i-1}(m_{2})T_{j-1}(m_{2})+\lambda
_{3}^{r}T_{i-1}(m_{3})T_{j-1}(m_{3})\right.   \label{2} \\
&&+2\dsum\limits_{\underset{k\neq 2,3}{k=1}}^{n}\lambda
_{k}^{r}T_{i-1}(m_{k})T_{j-1}(m_{k}));\ i=1,\ldots ,n;\ j=1,n  \notag
\end{eqnarray}%
and%
\begin{eqnarray}
\widetilde{a}_{ij}(r) &=&\frac{(-1)^{j}}{n-1}\left( 
\begin{array}{c}
\\ 
(\lambda _{2}^{r}T_{i-1}(m_{2})T_{j-1}(m_{2})+\lambda
_{3}^{r}T_{i-1}(m_{3})T_{j-1}(m_{3})) \\ 
\end{array}%
\right.   \label{3} \\
&&\left. +2\dsum\limits_{\underset{k\neq 2,3}{k=1}}^{n}\lambda
_{k}^{r}T_{i-1}(m_{k})T_{j-1}(m_{k})\right) ;i=1,\ldots ,n;\ j=2,\ldots ,n-1
\notag
\end{eqnarray}%
where$\ m_{k}=\frac{\lambda _{k}-a}{2b}$ for $\lambda _{k}$ is the $k$-th
eigenvalue of $\widetilde{A}_{n}$ see [14, p. 3]\ and $T_{j}(.)$ is the $j-$%
th degree Chebyshev polynomial of the first kind [11, p. 14]. \"{O}tele\c{s}
et al. computed integer powers of certain complex tridiagonal matrix%
\begin{equation*}
\widetilde{B}_{n}=\left\{ 
\begin{array}{l}
\widetilde{b}_{11}=\widetilde{b}_{nn}=a+b \\ 
\widetilde{b}_{12}=\widetilde{b}_{21}=\widetilde{b}_{n-1,n}=\widetilde{b}%
_{n,n-1}=b \\ 
tridiag(b,a,b),\ other.%
\end{array}%
\right. [9,\ p.67]
\end{equation*}%
as%
\begin{equation}
\widetilde{b}_{ij}^{r}=\dsum\limits_{k=1}^{n}f_{k}\lambda _{k}^{r}T_{\frac{%
2i-1}{2}}\left( \frac{\lambda _{k}-a}{2b}\right) T_{\frac{2j-1}{2}}\left( 
\frac{\lambda _{k}-a}{2b}\right) \text{; }i,j=1,\ldots ,n  \label{4}
\end{equation}%
where $T_{j}(.)$ is the $j-$th degree Chebyshev polynomial of the first kind
[11, p. 14], $\lambda _{k}$\ is the $k$-th eigenvalue of the matrix $%
\widetilde{B}_{n}$ and 
\begin{equation*}
f_{k}=\left\{ 
\begin{array}{l}
\frac{2}{n},\ \text{if }k=1,\ldots ,n-1 \\ 
\frac{1}{n},\text{ if }k=n.%
\end{array}%
\right. 
\end{equation*}

Let

\begin{equation}
A_{n}:=\left( 
\begin{array}{cccccc}
&  &  &  & 2b & a \\ 
&  &  & -b & a & b \\ 
&  & {\mathinner{\mkern2mu\raise1pt\hbox{.}\mkern2mu
\raise4pt\hbox{.}\mkern2mu\raise7pt\hbox{.}\mkern1mu}} & a & -b &  \\ 
& -b & {\mathinner{\mkern2mu\raise1pt\hbox{.}\mkern2mu
\raise4pt\hbox{.}\mkern2mu\raise7pt\hbox{.}\mkern1mu}} & {%
\mathinner{\mkern2mu\raise1pt\hbox{.}\mkern2mu
\raise4pt\hbox{.}\mkern2mu\raise7pt\hbox{.}\mkern1mu}} &  &  \\ 
b & a & -b &  &  &  \\ 
a & 2b &  &  &  & 
\end{array}%
\right)  \label{5}
\end{equation}%
and%
\begin{equation}
B_{n}:=\left( 
\begin{array}{cccccc}
&  &  &  & b & a+b \\ 
&  &  & b & a & b \\ 
&  & b & a & b &  \\ 
& {\mathinner{\mkern2mu\raise1pt\hbox{.}\mkern2mu
\raise4pt\hbox{.}\mkern2mu\raise7pt\hbox{.}\mkern1mu}} & {%
\mathinner{\mkern2mu\raise1pt\hbox{.}\mkern2mu
\raise4pt\hbox{.}\mkern2mu\raise7pt\hbox{.}\mkern1mu}} & {%
\mathinner{\mkern2mu\raise1pt\hbox{.}\mkern2mu
\raise4pt\hbox{.}\mkern2mu\raise7pt\hbox{.}\mkern1mu}} &  &  \\ 
b & a & b &  &  &  \\ 
a+b & b &  &  &  & 
\end{array}%
\right)  \label{6}
\end{equation}%
be the anti-tridiagonal matrices, where $b\neq 0$ and $a,b\in 
\mathbb{C}
$. In this paper, we obtain the integer powers of the $n\times n$ complex
anti-tridiagonal matrices in (5) and (6). We also get the complex
factorizations of Fibonacci polynomials, Fibonacci and Pell numbers using
the eigenvalues of the matrices $A_{n}$ and $B_{n}$.

\section{General Expression of $A_{n}^{r}$}

In this section, we obtain a general expression for the entries of the $r$%
-th powers of an $n\times n$ complex anti-tridiagonal matrix $A_{n}$ and $%
B_{n}$\ in (5) and (6) where if $n$ is even, $r\in \mathbb{N}$ or $n$ is
odd, $r\in \mathbb{Z}$.

\begin{lemma}
Let $a,0\neq b\in 
\mathbb{C}
,$ $n\in \mathbb{N}\ $and%
\begin{equation}
\widetilde{A}_{n}=\left( 
\begin{array}{ccccccc}
a & 2b &  &  &  &  &  \\ 
b & a & b &  &  &  &  \\ 
& -b & a & -b &  &  &  \\ 
&  & -b & \ddots & \ddots &  &  \\ 
&  &  & \ddots & a & -b &  \\ 
&  &  &  & -b & a & b \\ 
&  &  &  &  & 2b & a%
\end{array}%
\right) ,  \label{7}
\end{equation}%
\begin{equation}
\widetilde{B}_{n}=\left( 
\begin{array}{ccccccc}
a+b & b &  &  &  &  &  \\ 
b & a & b &  &  &  &  \\ 
& b & a & b &  &  &  \\ 
&  & b & \ddots & \ddots &  &  \\ 
&  &  & \ddots & a & b &  \\ 
&  &  &  & b & a & b \\ 
&  &  &  &  & b & a+b%
\end{array}%
\right)  \label{8}
\end{equation}%
and%
\begin{equation}
J_{n}=\left( 
\begin{array}{ccccc}
&  &  &  & 1 \\ 
&  &  & 1 &  \\ 
&  & {\mathinner{\mkern2mu\raise1pt\hbox{.}\mkern2mu
\raise4pt\hbox{.}\mkern2mu\raise7pt\hbox{.}\mkern1mu}} &  &  \\ 
& 1 &  &  &  \\ 
1 &  &  &  & 
\end{array}%
\right) .  \label{9}
\end{equation}%
Then%
\begin{equation*}
A_{n}=J_{n}\widetilde{A}_{n}=\widetilde{A}_{n}J_{n}
\end{equation*}%
and%
\begin{equation*}
B_{n}=J_{n}\widetilde{B}_{n}=\widetilde{B}_{n}J_{n}
\end{equation*}
\end{lemma}

\begin{proof}
See [7].
\end{proof}

\begin{lemma}
\bigskip Let $A_{n}$ be an $n\times n$ complex anti-tridiagonal matrix given
by (5). Then%
\begin{equation}
A_{n}^{r}=\left\{ 
\begin{array}{l}
J_{n}\widetilde{A}_{n}^{r},\ r\text{ is odd} \\ 
\widetilde{A}_{n}^{r},\ \ \ \ r\text{ is even.}%
\end{array}%
\right.  \label{10}
\end{equation}
\end{lemma}

\begin{proof}
We will show by induction on $r$. The case $r=1$ is clear. Suppose that the
equality (10) is true for $r>1.$ Now let us show the equality (10) is true
for $r+1.$ By the induction hypothesis we have%
\begin{equation*}
A_{n}^{r+1}=\left\{ 
\begin{array}{l}
\widetilde{A}_{n}J_{n}\widetilde{A}_{n}^{r},\ \ \ \ \ r+1\text{ is odd} \\ 
\widetilde{A}_{n}J_{n}\widetilde{A}_{n}^{r}J_{n},\ r+1\text{ is even.}%
\end{array}%
\right.
\end{equation*}%
Since $A_{n}=\widetilde{A}_{n}J_{n},$ we obtain%
\begin{equation*}
A_{n}^{r+1}=\left\{ 
\begin{array}{l}
J_{n}\widetilde{A}_{n}^{r+1},\ r+1\text{ is odd} \\ 
\widetilde{A}_{n}^{r+1},\ \ \ \ r+1\text{ is even.}%
\end{array}%
\right.
\end{equation*}
\end{proof}

The same proof can be done easily for the matrix $B_{n}$.

\begin{theorem}
\bigskip Let $A_{n}$ be an $n\times n$ complex anti-tridiagonal matrix given
by (5). If $n$ is odd, then the r-th power of $A_{n}$ is 
\begin{eqnarray}
a_{n-i+1,j}^{r} &=&\frac{1}{2n-2}\left( \lambda
_{2}^{r}T_{n-i}(m_{2})T_{j-1}(m_{2})+\lambda
_{3}^{r}T_{n-i}(m_{3})T_{j-1}(m_{3})\right.  \notag \\
&&+2\dsum\limits_{\underset{k\neq 2,3}{k=1}}^{n}\lambda
_{k}^{r}T_{n-i}(m_{k})T_{j-1}(m_{k}));\ i=1,\ldots ,n;\ j=1,n  \label{11}
\end{eqnarray}%
and%
\begin{eqnarray}
a_{n-i+1,j}^{r} &=&\frac{(-1)^{j}}{n-1}\left( 
\begin{array}{c}
\\ 
\lambda _{2}^{r}T_{n-i}(m_{2})T_{j-1}(m_{2})+\lambda
_{3}^{r}T_{n-i}(m_{3})T_{j-1}(m_{3}) \\ 
\end{array}%
\right.  \notag \\
&&\left. +2\dsum\limits_{\underset{k\neq 2,3}{k=1}}^{n}\lambda
_{k}^{r}T_{n-i}(m_{k})T_{j-1}(m_{k})\right)  \label{12}
\end{eqnarray}%
for $i=1,\ldots ,n;\ j=2,\ldots ,n-1$\ and if $n$ is even, then the $r$-th
power of $A_{n}$ is%
\begin{eqnarray}
a_{i,j}^{r} &=&\frac{1}{2n-2}\left( \lambda
_{2}^{r}T_{i-1}(m_{2})T_{j-1}(m_{2})+\lambda
_{3}^{r}T_{i-1}(m_{3})T_{j-1}(m_{3})\right.  \notag \\
&&+2\dsum\limits_{\underset{k\neq 2,3}{k=1}}^{n}\lambda
_{k}^{r}T_{i-1}(m_{k})T_{j-1}(m_{k}));\ i=1,\ldots ,n;\ j=1,n  \label{13}
\end{eqnarray}%
and%
\begin{eqnarray}
a_{i,j}^{r} &=&\frac{(-1)^{j}}{n-1}\left( 
\begin{array}{c}
\\ 
\lambda _{2}^{r}T_{i-1}(m_{2})T_{j-1}(m_{2})+\lambda
_{3}^{r}T_{i-1}(m_{3})T_{j-1}(m_{3}) \\ 
\end{array}%
\right.  \notag \\
&&\left. +2\dsum\limits_{\underset{k\neq 2,3}{k=1}}^{n}\lambda
_{k}^{r}T_{i-1}(m_{k})T_{j-1}(m_{k})\right)  \label{14}
\end{eqnarray}%
for $i=1,\ldots ,n;\ j=2,\ldots ,n-1.$
\end{theorem}

\begin{proof}
Let $A_{n}^{r}=(a_{ij}^{r})$ and $U=\widetilde{A}_{n}^{r}$. We obtained the
eigenvalues of\ $\widetilde{A}_{n}$ as%
\begin{equation*}
\lambda _{k}=a+2b\cos \left( \frac{(k-1)\pi }{n-1}\right) ,\ k=1,\ldots ,n\
(see\ [14,p.3])
\end{equation*}%
and the entries of the matrix $U$\ as%
\begin{eqnarray}
u_{ij}(r) &=&\frac{1}{2n-2}\left( \lambda
_{2}^{r}T_{i-1}(m_{2})T_{j-1}(m_{2})+\lambda
_{3}^{r}T_{i-1}(m_{3})T_{j-1}(m_{3})\right.   \notag \\
&&+2\dsum\limits_{\underset{k\neq 2,3}{k=1}}^{n}\lambda
_{k}^{r}T_{i-1}(m_{k})T_{j-1}(m_{k}));\ i=1,\ldots ,n;\ j=1,n  \label{15}
\end{eqnarray}%
and for $i=1,\ldots ,n$ and$\ j=2,\ldots ,n-1$%
\begin{eqnarray}
u_{ij}(r) &=&\frac{(-1)^{j}}{n-1}\left( 
\begin{array}{c}
\\ 
\lambda _{2}^{r}T_{i-1}(m_{2})T_{j-1}(m_{2})+\lambda
_{3}^{r}T_{i-1}(m_{3})T_{j-1}(m_{3}) \\ 
\end{array}%
\right.   \notag \\
&&\left. +2\dsum\limits_{\underset{k\neq 2,3}{k=1}}^{n}\lambda
_{k}^{r}T_{i-1}(m_{k})T_{j-1}(m_{k})\right)   \label{16}
\end{eqnarray}%
where $m_{k}=\frac{\lambda _{k}-a}{2b}\ $(see [14, p.]) and $T_{s}(.)$ is
the $s-$th degree Chebyshev polynomial of the first kind [13, p. 14]. If $r$
is even, then the equalities (15) and (16) are valid by the equality (10).
Let $r$ be odd number. If we multiply the equalities (15) and (16) by $J_{n}$
from left side, then we have%
\begin{equation*}
(J_{n}\widetilde{A}_{n}^{r})_{i,k}=\dsum%
\limits_{s=1}^{n}(J_{n})_{i,s}u_{s,k}(r)=u_{n-i+1,k}(r);\ k=1,\ldots ,n.
\end{equation*}%
Hence we obtain%
\begin{eqnarray*}
a_{n-i+1,j}^{r} &=&u_{n-i+1,j}(r) \\
&=&\frac{1}{2n-2}\left( \lambda _{2}^{r}T_{n-i}(m_{2})T_{j-1}(m_{2})+\lambda
_{3}^{r}T_{n-i}(m_{3})T_{j-1}(m_{3})\right.  \\
&&+2\dsum\limits_{\underset{k\neq 2,3}{k=1}}^{n}\lambda
_{k}^{r}T_{n-i}(m_{k})T_{j-1}(m_{k}));\ i=1,\ldots ,n;\ j=1,n
\end{eqnarray*}%
and%
\begin{eqnarray*}
a_{n-i+1,j}^{r} &=&u_{n-i+1,j}(r) \\
&=&\frac{(-1)^{j}}{n-1}\left( 
\begin{array}{c}
\\ 
\lambda _{2}^{r}T_{n-i}(m_{2})T_{j-1}(m_{2})+\lambda
_{3}^{r}T_{n-i}(m_{3})T_{j-1}(m_{3}) \\ 
\end{array}%
\right.  \\
&&\left. +2\dsum\limits_{\underset{k\neq 2,3}{k=1}}^{n}\lambda
_{k}^{r}T_{n-i}(m_{k})T_{j-1}(m_{k})\right) ;i=1,\ldots ,n;\ j=2,\ldots ,n-1.
\end{eqnarray*}
\end{proof}

\begin{theorem}
\bigskip Let $B_{n}$ be an $n\times n$ complex anti-tridiagonal matrix given
by (6), $\widetilde{V}=\widetilde{B}_{n}^{r}$ and $V=B_{n}^{r}.$ Then the
entries of $V$ are%
\begin{equation}
v_{ij}(r)=\left\{ 
\begin{array}{l}
\dsum\limits_{k=1}^{n}f_{k}\lambda _{k}^{r}T_{\frac{2i-1}{2}}\left( \frac{%
\lambda _{k}-a}{2b}\right) T_{\frac{2j-1}{2}}\left( \frac{\lambda _{k}-a}{2b}%
\right) ,\text{ if r is even} \\ 
\dsum\limits_{k=1}^{n}f_{k}\lambda _{k}^{r}T_{\frac{2(n-i)+1}{2}}\left( 
\frac{\lambda _{k}-a}{2b}\right) T_{\frac{2j-1}{2}}\left( \frac{\lambda
_{k}-a}{2b}\right) ,\text{ if r is odd}%
\end{array}%
\right.  \label{17}
\end{equation}%
where $i,j=1,\ldots ,n$ and 
\begin{equation*}
\lambda _{k}=a+2b\cos \left( \frac{(k-1)\pi }{n}\right) ,\ k=1,\ldots ,n
\end{equation*}%
are the eigenvalues of the matrix $\widetilde{B}_{n}.$
\end{theorem}

\begin{proof}
\bigskip Since $B_{n}^{r}=\widetilde{B}_{n}^{r}$ from (10) for $r$ is even,
the equality (17) is valid. Let $r$ be odd number. If we multiply the
equality (4) by $J_{n}$ from the left side, then we have%
\begin{eqnarray*}
v_{i,j} &=&(J_{n}\widetilde{B}_{n}^{r})_{i,j}=\dsum%
\limits_{s=1}^{n}(J_{n})_{i,s}\widetilde{v}_{s,k}(r) \\
&=&\widetilde{v}_{n-i+1,j}(r) \\
&=&\dsum\limits_{k=1}^{n}f_{k}\lambda _{k}^{r}T_{\frac{2(n-i)+1}{2}}\left( 
\frac{\lambda _{k}-a}{2b}\right) T_{\frac{2j-1}{2}}\left( \frac{\lambda
_{k}-a}{2b}\right) .
\end{eqnarray*}
\end{proof}

\begin{corollary}
Let the matrix $A_{n}$ be as in (5). Then the eigenvalues of $A_{n}$ are%
\begin{equation}
\lambda _{k}=(-1)^{k}\left( a+2b\cos \left( \frac{(k-1)\pi }{n-1}\right)
\right) ,k=1,\ldots ,n.  \label{18}
\end{equation}
\end{corollary}

\begin{proof}
See [3, p.574] and [14, p. 5].
\end{proof}

\begin{corollary}
Let the matrix $B_{n}$ be as in (6). Then the eigenvalues of $B_{n}$ are%
\begin{equation}
\lambda _{k}=(-1)^{k-1}\left( a+2b\cos \left( \frac{(k-1)\pi }{n}\right)
\right)  \label{19}
\end{equation}%
where $k=1,\ldots ,n.$
\end{corollary}

\begin{proof}
See [5, p.574] and [14, p. ].
\end{proof}

\section{\textbf{Numerical \ examples}}

Considering the Eqs. (11-14), we can find the arbitrary integer powers of
the $n\times n$ complex anti-tridiagonal matrix $A_{n}$ in (5), where $n$ is
positive odd integer.

\begin{example}
Let $n=3,r=3,a=1$ and $b=3.$\ Since%
\begin{equation*}
\widetilde{J}=diag(\lambda _{1},\lambda _{2},\lambda
_{3})=diag(a,a+2b,a-2b)=diag(1,7,-5)
\end{equation*}%
and%
\begin{equation*}
\widetilde{A}_{3}^{3}=(u_{ij}(3))=\left( 
\begin{array}{ccc}
55 & 234 & 54 \\ 
117 & 109 & 117 \\ 
54 & 234 & 55%
\end{array}%
\right) ,
\end{equation*}%
we obtain%
\begin{equation*}
A_{3}^{3}=J_{3}\widetilde{A}_{3}^{3}=\left( 
\begin{array}{ccc}
54 & 234 & 55 \\ 
117 & 109 & 117 \\ 
55 & 234 & 54%
\end{array}%
\right) .
\end{equation*}
\end{example}

\begin{example}
If $n=5,r=4,a=1$ and $b=3,$ then 
\begin{eqnarray*}
\widetilde{J} &=&diag(\lambda _{1},\lambda _{2},\lambda _{3},\lambda
_{4},\lambda _{5}) \\
&=&diag(1,7,-5,1+3\sqrt{2},1-3\sqrt{2})
\end{eqnarray*}%
and since%
\begin{equation*}
\widetilde{A}_{5}^{4}=(u_{ij}(4))=\left( 
\begin{array}{rrrrr}
595 & 672 & -756 & 216 & 162 \\ 
336 & 973 & -444 & 540 & 108 \\ 
-378 & -444 & 757 & -444 & -378 \\ 
108 & 540 & -444 & 973 & 336 \\ 
162 & 216 & -756 & 672 & 595%
\end{array}%
\right) ,
\end{equation*}%
we have%
\begin{equation*}
A_{5}^{4}=J_{5}\widetilde{A}_{5}^{4}=\left( 
\begin{array}{rrrrr}
162 & 216 & -756 & 672 & 595 \\ 
108 & 540 & -444 & 973 & 336 \\ 
-378 & -444 & 757 & -444 & -378 \\ 
336 & 973 & -444 & 540 & 108 \\ 
595 & 672 & -756 & 216 & 162%
\end{array}%
\right) .
\end{equation*}
\end{example}

\section{Complex Factorizations}

The well-known the Fibonacci polynomials$\ F(x)=\{F_{n}(x)\}_{n=1}^{\infty }$
are defined by $F_{n}(x)=xF_{n-1}(x)+F_{n-2}(x)$ with initial conditions $%
F_{0}(x)=0$ and $F_{1}(x)=1.$ For $x=1$ and $x=2,$\ then we obtain the
Fibonacci and Pell numbers as%
\begin{equation*}
F_{n}(1)=\{0,1,1,2,3,5,8,\ldots \}
\end{equation*}%
and%
\begin{equation*}
F_{n}(2)=\{0,1,2,5,12,29,\ldots \},
\end{equation*}%
respectively.

\begin{theorem}
\bigskip Let $A_{n}$ be an $n\times n$ complex anti-tridiagonal matrix given
by (5). If $a:=x$ and $b:=\mathbf{i}$, then%
\begin{equation*}
\det (A_{n})=\left\{ 
\begin{array}{r}
\ \ \ (x^{2}+4)F_{n-1}(x),\ n\equiv 0\ or\ 1\ \func{mod}4 \\ 
-(x^{2}+4)F_{n-1}(x),\ n\equiv 2\ or\ 3\ \func{mod}4%
\end{array}%
\right.
\end{equation*}%
where $\mathbf{i}=\sqrt{-1}.$
\end{theorem}

\begin{proof}
Since $\widetilde{A}_{n}=J_{n}A_{n},$%
\begin{equation*}
\det (\widetilde{A}_{n})=x^{2}D_{n-2}+4xD_{n-3}+4D_{n-4}
\end{equation*}%
where $D_{n}=\det (tridiag_{n}(-\mathbf{i},x,-\mathbf{i}))$ and%
\begin{equation*}
\det (tridiag_{n}(-\mathbf{i},x,-\mathbf{i}))=F_{n+1}(x)\ (see\ [14,\ p.]),
\end{equation*}%
we arrive at%
\begin{eqnarray*}
\det (\widetilde{A}_{n}) &=&x^{2}F_{n-1}(x)+4xF_{n-2}(x)+4F_{n-3}(x) \\
&=&x^{2}(xF_{n-2}(x)+F_{n-3}(x))+4xF_{n-2}(x)+4F_{n-3}(x) \\
&=&(x^{2}+4)(xF_{n-2}(x)+F_{n-3}(x))=(x^{2}+4)F_{n-1}(x).
\end{eqnarray*}%
Since%
\begin{equation*}
\det (J_{n})=\left\{ 
\begin{array}{r}
1,\ \ n\equiv 0\ or\ 1\func{mod}4 \\ 
-1,n\equiv 2\ or\ 3\func{mod}4,%
\end{array}%
\right.
\end{equation*}%
the proof of theorem is completed.
\end{proof}

\begin{corollary}
Let $A_{n}$ be an $n\times n$ complex anti-tridiagonal matrix given by (5).
If $a:=x$ and $b:=\mathbf{i}$, then the complex factorization of generalized
Fibonacci-Pell numbers is the following form:%
\begin{equation*}
F_{n-1}(x)=\frac{\alpha }{x^{2}+4}\dprod\limits_{k=1}^{n}\left( x+2\mathbf{i}%
\cos \left( \frac{(k-1)\pi }{n-1}\right) \right)
\end{equation*}%
where%
\begin{equation*}
\alpha =\left\{ 
\begin{array}{l}
(-1)^{n},\ if\ n\equiv 0\ or\ 1\ \func{mod}4 \\ 
(-1)^{n-1},if\ n\equiv 2\ or\ 3\ \func{mod}4.%
\end{array}%
\ \right.
\end{equation*}
\end{corollary}

\begin{proof}
Since the eigenvalues of the matrix $A_{n}$ from (18)%
\begin{equation*}
\lambda _{j}=(-1)^{j}\left( x+2\mathbf{i}\cos \left( \frac{(j-1)\pi }{n-1}%
\right) \right) ,\ j=1,\ldots ,n,
\end{equation*}%
the determinant of the matrix $A_{n}$ can be obtained as%
\begin{equation}
\det (A_{n})=(-1)^{n}\dprod\limits_{k=1}^{n}\left( x+2\mathbf{i}\cos \left( 
\frac{(k-1)\pi }{n-1}\right) \right) .  \label{20}
\end{equation}%
By considering (20) and Theorem 9, the complex factorization of generalized
Fibonacci-Pell numbers is obtained.
\end{proof}

\begin{theorem}
Let $B_{n}$ be an $n\times n$ complex anti-tridiagonal matrix given by (6).
If $a:=1$ and $b:=\mathbf{i}$, then 
\begin{equation}
\det (B_{n})=\left\{ 
\begin{array}{c}
(1+2\mathbf{i})F_{n},\ \ n\equiv 0\ or\ 1\func{mod}4 \\ 
-(1+2\mathbf{i})F_{n},n\equiv 2\ or\ 3\func{mod}4%
\end{array}%
\right.  \label{21}
\end{equation}%
and if $a:=2$ and $b:=\mathbf{i}$, then 
\begin{equation}
\det (B_{n})=\left\{ 
\begin{array}{c}
(2+2\mathbf{i})P_{n},\ \ n\equiv 0\ or\ 1\func{mod}4 \\ 
-(2+2\mathbf{i})P_{n},n\equiv 2\ or\ 3\func{mod}4%
\end{array}%
\right.  \label{22}
\end{equation}%
where $\mathbf{i}=\sqrt{-1}$ and $F_{n}$ and $P_{n}$ denote the nth
Fibonacci and Pell numbers, respectively.
\end{theorem}

\begin{proof}
Applying Laplace expansion according to the first two and last two rows of
the determinant of $\widetilde{B}_{n}$, we have%
\begin{eqnarray}
\det (\widetilde{B}_{n}) &=&(a+b)^{2}\det (tridiag_{n-2}(b,a,b))  \label{23}
\\
&&-2b^{2}(a+b)\det (tridiag_{n-3}(b,a,b))  \notag \\
&&+b^{4}\det (tridiag_{n-4}(b,a,b))\ (see\ [9,\ p.]).  \notag
\end{eqnarray}%
If we take $a:=1$ and $b:=\mathbf{i}$ in (23), then we get%
\begin{eqnarray*}
\det (\widetilde{B}_{n}) &=&(1+\mathbf{i})^{2}\det (tridiag_{n-2}(\mathbf{i}%
,1,\mathbf{i})) \\
&&+2(1+\mathbf{i})\det (tridiag_{n-3}(\mathbf{i},1,\mathbf{i})) \\
&&+\det (tridiag_{n-4}(\mathbf{i},1,\mathbf{i}))\ (see\ [9,\ p.]) \\
&=&(1+\mathbf{i})^{2}F_{n-1}+2(1+\mathbf{i})F_{n-2}+F_{n-3} \\
&=&(1+2\mathbf{i})F_{n}.
\end{eqnarray*}%
Since%
\begin{equation*}
B_{n}=J_{n}\widetilde{B}_{n},
\end{equation*}%
we obtain (21).

Similar to the above, we can easily obtain Pell numbers.
\end{proof}

\begin{corollary}
Let $B_{n}$ be an $n\times n$ complex anti-tridiagonal matrix given by (6).
If $a:=1\ $and $b:=\mathbf{i}$, then the complex factorizations of Fibonacci
numbers are%
\begin{equation*}
F_{n}=(-1)^{n}\left\{ 
\begin{array}{l}
\ \ \dprod\limits_{k=2}^{n-1}\left( 1+2\mathbf{i}\cos \left( \frac{(k-1)\pi 
}{n}\right) \right) ,n\ is\ even \\ 
-\dprod\limits_{k=2}^{n}\left( 1+2\mathbf{i}\cos \left( \frac{(k-1)\pi }{n}%
\right) \right) ,n\ is\ odd%
\end{array}%
\right.
\end{equation*}%
and if $a:=2\ $and $b:=\mathbf{i,}$ then the complex factorizations of Pell
numbers are%
\begin{equation*}
P_{n}=(-1)^{n}\left\{ 
\begin{array}{l}
\ \ \dprod\limits_{k=2}^{n}\left( 2+2\mathbf{i}\cos \left( \frac{(k-1)\pi }{n%
}\right) \right) ,\ n\ is\ even \\ 
-\dprod\limits_{k=2}^{n}\left( 2+2\mathbf{i}\cos \left( \frac{(k-1)\pi }{n}%
\right) \right) ,\ n\ is\ odd.%
\end{array}%
\right.
\end{equation*}
\end{corollary}

\begin{proof}
Let $a:=1\ $and $b:=\mathbf{i.}$ Since the eigenvalues of the matrix $B_{n}$
are%
\begin{equation*}
\lambda _{k}=(-1)^{n-1}\left( 1+2\mathbf{i}\cos \left( \frac{(k-1)\pi }{n}%
\right) \right)
\end{equation*}%
for $k=1,\ldots ,n$\ and the determinant of the matrix $B_{n}$ is equal to
multiplication of its eigenvalues, we get%
\begin{eqnarray*}
F_{n} &=&\frac{1}{1+2\mathbf{i}}\left\{ 
\begin{array}{c}
\det (B_{n}),\ \ n\equiv 0\ or\ 1\func{mod}4 \\ 
-\det (B_{n}),n\equiv 2\ or\ 3\func{mod}4%
\end{array}%
\right. \\
&=&\frac{1}{1+2\mathbf{i}}\left\{ 
\begin{array}{l}
\ \ \ (-1)^{n}\dprod\limits_{k=1}^{n}\left( 1+2\mathbf{i}\cos \left( \frac{%
(k-1)\pi }{n}\right) \right) ,\ n\ is\ even \\ 
(-1)^{n+1}\dprod\limits_{k=1}^{n}\left( 1+2\mathbf{i}\cos \left( \frac{%
(k-1)\pi }{n}\right) \right) ,\ n\ is\ odd%
\end{array}%
\right. \\
&=&(-1)^{n}\left\{ 
\begin{array}{l}
\ \ \dprod\limits_{k=2}^{n}\left( 1+2\mathbf{i}\cos \left( \frac{(k-1)\pi }{n%
}\right) \right) ,\ n\ is\ even \\ 
-\dprod\limits_{k=2}^{n}\left( 1+2\mathbf{i}\cos \left( \frac{(k-1)\pi }{n}%
\right) \right) ,\ n\ is\ odd.%
\end{array}%
\right.
\end{eqnarray*}%
We can easily obtain Pell numbers similarly for $a:=2\ $and $b:=\mathbf{i.}$

Thus, the proof is completed.
\end{proof}

\textbf{Acknowledgement. }The authors are partially supported by TUBITAK and
the Office of Sel\c{c}uk University Research Project (BAP).

\end{document}